\documentclass{amsart}
\usepackage{amsmath, amssymb, amsthm, tikz}

  \tikzstyle{vertex} = [circle,fill=blue!20]
  \tikzstyle{Svertex} = [circle,fill=blue!20,minimum size = 2mm,inner sep = 0pt]
  \tikzstyle{edge} = [draw,thick,-]
  \tikzstyle{thickedge} = [draw,line width=0.75mm,-]
  \tikzstyle{arc} = [draw,thick,-,dashed]
  \tikzstyle{label} = [font=\normalsize,rectangle,fill = white, text = black, fill opacity = 1, text opacity = 1, scale = 0.88]
  \tikzstyle{label2} = [font=\normalsize,rectangle,fill = gray, text = black, fill opacity = 1, text opacity = 1, scale = 0.88]

\newtheorem{theorem}{Theorem}[section]
\newtheorem{lemma}[theorem]{Lemma}
\newtheorem{corollary}[theorem]{Corollary}
\newtheorem{conjecture}[theorem]{Conjecture}
\newtheorem{obs}[theorem]{Observation}

\newcommand{\spc}{\hspace{0.08in}}

\newcommand{\mg}[1]{\textcolor{black!100}{#1}}

\begin{document}

\title{On Clique Immersions in Line Graphs}

\author{Michael Guyer}
\address{Department of Mathematics and Statistics, Auburn University, Auburn, AL 36849}
\email{mdg0036@auburn.edu}

\author{Jessica McDonald}
\address{Department of Mathematics and Statistics, Auburn University, Auburn, AL 36849}
\email{mcdonald@auburn.edu}
\thanks{Supported in part by NSF grant DMS-1600551.}

\date{}

\begin{abstract}
We prove that if $L(G)$ immerses $K_t$ then $L(mG)$ immerses $K_{mt}$, where $mG$ is the graph obtained from $G$ by replacing each edge in $G$ with a parallel edge of multiplicity $m$. This implies that when $G$ is a simple graph, $L(mG)$ satisfies a conjecture of Abu-Khzam and Langston.
We also show that when $G$ is a line graph, $G$ has a $K_t$-immersion iff $G$ has a $K_t$-minor whenever $t\leq 4$, but this equivalence fails in both directions when \mg{$t\geq5$}.
\end{abstract}

\maketitle

\section{Introduction}

In this paper, a graph is permitted to have parallel edges (but no loops), unless it is explicitly said to be simple.

Immersion is a containment relation in graphs that is similar to the well-known minor relation, but is incomparable. Formally, a pair of adjacent edges $uv$ and $vw$ in a graph
are \emph{split off} from their common vertex $v$ by deleting the edges
$uv$ and $vw$, and adding the edge $uw$ (unless it forms a loop, i.e., $u=w$). Given graphs $G, H$, we say that $G$ has an \emph{$H$-immersion} if $H$ can be obtained from a subgraph of $G$ by splitting off pairs of
edges and removing isolated vertices. In comparison, $G$ has an \emph{$H$-minor} (\emph{topological $H$-minor}) if $H$ can be obtained from a subgraph of $G$ by contracting edges (suppressing vertices of degree two). The existence of a topological $H$-minor immediately implies both an $H$-minor and an $H$-immersion.  Immersions have gained considerable interest in the last number of years, with a major factor being the publication of Robertson and Seymour's \cite{RS23} proof that graphs are well-quasi-ordered by immersion (see for example \cite{bustamente}, \cite{dvorak}, \cite{gauthier}). An immersion conjecture of particular interest relates the ability to immerse a large clique in $G$ to the chromatic number $\chi(G)$ of $G$ (i.e. the minimum number of colours needed to assign colours to the vertices of $G$ so that adjacent vertices receive different colours).

\begin{conjecture} \label{AKL} \emph{(Abu-Khzam and Langston \cite{AKL})} For any graph $G$,
$$\chi(G)\geq t \spc\Rightarrow\spc \textrm{$G$ has a $K_t$-immersion}.$$
\end{conjecture}

Conjecture \ref{AKL} is an immersion-analog of Hadwiger's \cite{hadwiger} famous conjecture from the 1940s, namely that for any graph $G$, if $\chi(G)\geq t$ then $G$ has a $K_t$-minor.
Hadwiger's Conjecture is known to be true up to $t=6$, where it is equivalent to the Four Color Theorem (Robertson, Seymour and Thomas \cite{RST1}), while Conjecture \ref{AKL} is known to be true up to $t=7$ (DeVos, Kawarabayashi, Mohar, and Okamura \cite{DKMO}).

One class of graphs for which the Abu-Khzam--Langston Conjecture has not yet been verified (although Hadwiger's Conjecture has been) is line graphs. Given a graph $G$, the \emph{line graph} of $G$, denoted by $L(G)$, is formed by defining a vertex for each edge in $G$, and joining two vertices by an edge if the corresponding edges in $G$ are adjacent. Note that this definition makes complete sense whether $G$ is simple or not (although the graph $L(G)$ is always simple). Line graphs form a strictly larger class than line graphs of simple graphs, as forbidden-subgraph characterizations by Bermond and Meyer \cite{bm} and Beineke \cite{beineke}, respectively, make plain.

In Section 2 we shall describe a useful reformulation of Conjecture \ref{AKL} in the case when $G$ is a line graph, introducing the notion of \emph{semi-edge-disjoint paths}.  We shall also discuss work of Reed and Seymour \cite{reedseymour}, who proved that Hadwiger's Conjecture holds for all line graphs. Reed and Seymour's proof immediately implies that Conjecture \ref{AKL} holds for line graphs of simple graphs, but not for all line graphs. It is worth noting that there is also a completely different proof by Thomassen \cite{thomassen} that Conjecture \ref{AKL} holds for line graphs of simple graphs, however the `simple' assumption is crucial in his argument. (Thomassen actually verified Haj\'{o}s' Conjecture for any line graph of a simple graph, proving that for such a $G$, $\chi(G)\geq t$ implies that $G$ contains a topological $K_t$-minor). Conjecture \ref{AKL} is still open for line graphs of non-simple graphs. However, in Section 3 of this paper, we close the case when the `non-simple graph' has constant edge-multiplicity, by proving the following result. In what follows, for any $m\geq 2$, let $mG$ be the graph obtained by replacing each edge $e$ in a graph $G$ with $m$ copies of $e$.

\begin{theorem}\label{mH}
Let $H$ be a graph such that $G=L(H)$ has a $K_t$-immersion. Then, for any $m\geq 2$, $L(mH)$ has a $K_{mt}$-immersion.
\end{theorem}

The result of Theorem \ref{mH} is tight. Taking $H = K_3$, we have that $G=L(H)=K_3$ which trivially immerses $K_3$ but no larger clique. Observe that for any $m\geq 2$, $L(m K_3) = K_{3m}$, which does not immerse $K_n$ for $n$ any larger than $3m$.

Let $H$ be a graph such that $G=L(H)$ satisfies Conjecture \ref{AKL} (for example, if $H$ is any simple graph). Say that $\chi(G)=t$  and that $G$ has a $K_t$-immersion. Then, for any $m\geq 2$,  Theorem \ref{mH} tells us that $L(mH)$ immerses $K_{mt}$. Note that $\chi(L(mH)) \leq mt$, as we may decompose $mH$ into $m$ disjoint copies of $H$ and use $m$ disjoint color sets of size $t$ to color them. Therefore, we get that $L(mG)$ satisfies Conjecture \ref{AKL}. That is, we get the following Corollary.

\begin{corollary}
Let $H$ be a graph such that $G=L(H)$ satisfies Conjecture \ref{AKL} (for example, if $H$ is any simple graph). Then, for any $m\geq 2$, $L(mH)$ satisfies Conjecture \ref{AKL}.
\end{corollary}

An analog of Theorem \ref{mH} for minors turns out to be trivial, as we shall remark in Section 3. While minors and immersions are incomparable in general, it is worth asking whether or not this incomparability holds when restricted to clique immersions and clique minors in line graphs.  In Section 4 we consider this question, and prove the following.

\begin{theorem}\label{minorimm} Let $G$ be a line graph. If $t\leq 4$, then $G$ has a $K_t$-immersion iff $G$ has a $K_t$-minor. When \mg{$t\geq5$} this equivalence fails in both directions: there exists a line graph with a \mg{$K_t$}-immersion but no \mg{$K_t$}-minor, and there exists a line graph with a \mg{$K_t$}-minor but no \mg{$K_t$}-immersion.
\end{theorem}

The main work in proving Theorem \ref{minorimm} is the $t=4$ case, where we in fact prove a characterization for $K_4$-immersions and $K_4$-minors in line graphs. The equivalence in Theorem \ref{minorimm} fails for $t=4$ if $G$ is not a line graph, and a simple example is provided in Section 4. In terms of the \mg{$t\geq5$} result for line graphs, we will see in Section 4 that there are examples for any $t\geq 5$ of a line graph with a $K_t$-minor but no $K_t$-immersion. \mg{With a bit more care, we will also see examples in the other direction, i.e. line graphs that immerse $K_t$ but contain no $K_t$-minor.}

\section{Semi-edge-disjoint paths}

An equivalent and often useful definition of immersion is that $G$ has an $H$-immersion if there is a 1-1 function $\phi: V(H)\to V(G)$ such that for each edge $uv\in E(H)$, there is a path $P_{uv}$ in $G$ joining vertices $\phi(u)$ and $\phi(v)$, and the paths $P_{uv}$ are pairwise edge-disjoint for all $uv\in E(H)$. In this context we call the vertices of $\{\phi(v): v\in V(H)\}$ the \emph{terminals} of the $H$-immersion.

Let $\mathcal{P}$ be a set of paths in a graph $G$. We say that $\mathcal{P}$ is \emph{semi-edge-disjoint} if, for any 2-edge-path $P_0=(v_1, e_1, v_2, e_2, v_3)$ in $G$, $P_0$ is a subpath of at most one path in $\mathcal{P}$. Note that paths can have multiple edges in common and still be semi-edge-disjoint, we are just requiring that adjacencies between edges are not repeated. In what follows, by a path from edge $e$ to edge $f$ in $G$ (or equivalently, a path between $e$ and $f$ in $G$), we mean a path in $G$ whose first edge is $e$ and whose last edge is $f$.

\begin{obs}\label{sed}
Let $G$ be a line graph with $G=L(H)$. Then $G$ has a $K_t$-immersion if and only $H$ contains a set of $t$ distinct edges and a path between each pair of these edges such that this set of paths is semi-edge-disjoint.
\end{obs}

\begin{proof} By the alternate definition of immersion above, $G$ has a $K_t$-immersion iff it has a set of $t$ distinct terminal vertices and a path between each pair of these vertices, such that this set $\mathcal{P}_G$ of paths are edge-disjoint. Such a set of $t$ terminal vertices in $G$ corresponds exactly to $t$ distinct edges in $H$, and such a set $\mathcal{P}_G$ of paths in $G$ corresponds exactly to a path between each pair of these $t$ edges in $H$, such that this set of paths $\mathcal{P}_H$ in $H$ are semi-edge-disjoint.
\end{proof}

We shall refer to the $t$ distinct edges in Observation \ref{sed} as the \emph{terminals} (or \emph{terminal edges}) of the corresponding $K_t$-immersion.

Note that for any graph $H$, $\chi(L(H))=\chi'(H)$, where $\chi'(G)$ is the chromatic index of $G$ (the minimum number of colours needed to assign colours to the edges of $G$ so that adjacent edges receive different colours). Hence, using Observation \ref{sed}, we can restate Conjecture \ref{AKL} for line graphs as follows.

\begin{conjecture}\label{AKLLG} \emph{(The Abu-Khzam--Langston Conjecture for line graphs)} For any graph $H$, if $\chi'(H)\geq t$, then $H$ contains a set of $t$ distinct edges and a path between each pair of these edges such that this set of paths is semi-edge disjoint.
\end{conjecture}

The analogous restatement of Hadwiger's Conjecture for line graphs is as follows.

\begin{theorem}\label{HLG} \emph{(Reed and Seymour \cite{reedseymour})} For any graph $H$, if $\chi'(H)\geq t$, then $H$ contains a set of $t$ connected subgraphs, each with at least one edge, that are pairwise edge-disjoint but pairwise have at least one vertex in common.
\end{theorem}

We remark that Conjecture \ref{AKLLG} and Theorem \ref{HLG} obviously hold whenever $\chi'(H)=\Delta(H)$. This is because $H$ necessarily contains a star with $\Delta(H)$ edges, the collection of which fit the criteria for both the set of terminal edges needed in Conjecture \ref{AKLLG} and the set of subgraphs needed for Theorem \ref{HLG} (in fact these edges correspond exactly to a $K_{\Delta(H)}$ in $L(H)$). When $H$ is a simple graph, there is also a straightforward (and long-known) argument to show that the result of Theorem \ref{HLG} follows from Vizing's Theorem \cite{vizing}, and we include it here for completeness.

\begin{lemma} Let $H$ be a simple graph with $\chi'(H)\geq t$. Then $H$ contains a set of $t$ connected subgraphs, each with at least one edge, that are pairwise edge-disjoint but pairwise have at least one vertex in common.
\end{lemma}

\begin{proof} From above, we need only consider $t\geq \Delta(H)+1$. Since $H$ is simple, Vizing's Theorem tells us that $\chi'(H)\leq \Delta(H)+1$, so in fact we need only consider $t=\Delta(H)+1$.
Suppose that $H$ is an edge-minimal counterexample. Then, in particular, $H$ is connected. Let $v$ be a vertex of degree $\Delta(H)$ in $H$. Suppose first that $v$ is a cut vertex. If each shore of the cut (including $v$) can be individually edge-coloured with $\Delta(H)$ colours, then so can $H$ (by possibly permuting colours around $v$). So, one of these smaller graphs has chromatic index at least $\Delta(H)+1$. Since this smaller graph is not a counterexample, it has the desired set of subgraphs, which are also subgraphs of $H$. So, we may assume that $v$ is not a cut vertex. Consider the $\Delta(H)+1$ subgraphs induced by: the $\Delta(H)$ edges incident to $v$, and all other edges in the graph. Since $v$ is not a cut vertex, these subgraphs satisfy the needed conditions.
\end{proof}

There does not seem to be an easy analog of the above argument to show that Conjecture \ref{AKLLG} holds whenever $H$ is simple. Instead, we must rely on the afore-mentioned work of Reed and Seymour \cite{reedseymour} or Thomassen \cite{thomassen}. While Thomassen's work depends on $H$ being simple, it is worth taking a closer look at exactly what is and what is not implied by \cite{reedseymour}. In \cite{reedseymour}, Reed and Seymour proved the following result.

\begin{theorem}\label{ReedSey}\emph{(Reed and Seymour, \cite{reedseymour})} For any graph $H$ with chromatic index $t > \Delta(H)$, there exist vertices $u,v,w$ such that the number of edges between $v$ and $w$ plus the number of edge-disjoint paths between $u$ and $\{v,w\}$ is at least $t$.
\end{theorem}

The above theorem does not apply to graphs $H$ with $\chi'(H)=\Delta(H)$, but as we saw above this is irrelevant for both Conjecture \ref{AKLLG} and Theorem \ref{HLG}. The paths provided between the vertices $u, v, w$ clearly do provide the $t$ subgraphs needed to establish Theorem \ref{HLG}. If $H$ is simple, then Theorem \ref{ReedSey} also works to establish Conjecture \ref{AKLLG}. To see this, note that when $H$ is simple there is at most one edge between $v$ and $w$ (say $vw$), and since we are assuming $\chi'(H)=\Delta(H)+1=t$, this means there are exactly $\Delta(H)$ edge-disjoint paths between $u$ and
$\{v,w\}$. By taking the first edge in each of these $\Delta(H)$ paths, along with the edge $vw$, we get a set of $t$ distinct edges and a path between each pair of these edges such that this set of paths is semi-edge disjoint. On the other hand however, if $H$ is not simple, then the $t$ paths guaranteed by Theorem \ref{ReedSey} may include two or more edges between $v$ and $w$. In such a case, one of the edge-disjoint paths to $u$ may need to be used by both of these edges in order to reach some edge incident to $u$. See for example Figure \ref{RScounter}, where the 3 edge-disjoint paths between $u$ and $\{v,w\}$ are not enough to provide a set of semi-edge-disjoint paths between all of the five darkened  edges. The problem is not our selection of the 5 terminal edges. Indeed, the line graph of the graph pictured in Figure \ref{RScounter} has no $K_5$-immersion at all. One of the easiest ways to see this is to use the following basic observation about immersions.

\begin{obs} If $G$ has a $K_t$-immersion, then it has at least $t$ vertices with degree at least $t-1$.
\end{obs}

Since the immersion operation of splitting off edges never increases the degree of a vertex, the above observation is immediate. In terms of Figure \ref{RScounter}, we see that this graph $H$ has only 4 edges whose \emph{edge-degree} (number of other edges they are adjacent to) is at least 4; only the 4 edges incident to $v$ have this property. Hence, the line graph of $H$ cannot have a $K_5$-immersion.

It is worth noting that the graph $H$ in Figure \ref{RScounter} can be easily generalized from a counterexample for $t=5$ to any other larger value of $t$ by simply adding more 2-edge-paths between $u$ and $v$.\\

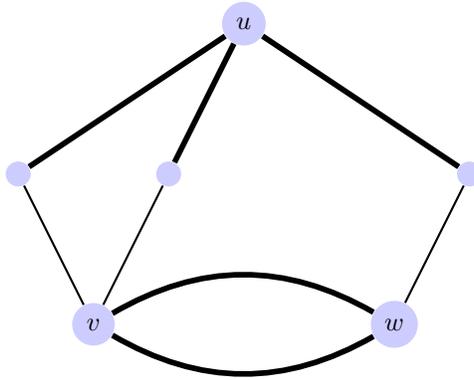
\begin{figure}
\centering
\begin{tikzpicture}

  \node[vertex] (v1) at (3,4) {$u$};
  \node[vertex] (v2) at (0,2)  {};
  \node[vertex] (v3) at (2,2)  {};
  \node[vertex] (v5) at (6,2)  {};
  \node[vertex] (v6) at (1,0) {$v$};
  \node[vertex] (v7) at (5,0)  {$w$};

  \foreach \from/\to in {v2/v6, v3/v6, v5/v7}
	\path[edge] (\from) -- (\to);

  \foreach \from/\to in {v1/v2,v1/v3,v1/v5}
	\path[thickedge] (\from) -- (\to);

\path[thickedge] (v6) to [bend right = 30] (v7);
\path[thickedge] (v6) to [bend left = 30] (v7);

\end{tikzpicture}
\caption{A graph $H$ where $L(H)$ contains $K_5$ as a minor but not as an immersion.}
\label{RScounter}
\end{figure}

\section{Graphs with constant edge multiplicity}

Before we prove Theorem \ref{mH}, it will be useful to have the following lemma.

\begin{lemma}\label{concat}
Let $\mathcal{P} = \{P_1,P_2,...,P_k\}$ and $\mathcal{Q} = \{Q_1,Q_2,...,Q_k\}$ be two semi-edge-disjoint sets of paths. Suppose that the last edge in $P_i$  is the first edge in $Q_i$, and that for all $i,j$, $|E(P_i) \cap E(Q_j)| \leq 1$. Then the set of paths formed by concatenating $P_i$ and $Q_i$ for each $i$ (leaving out the duplicate edge) is semi-edge-disjoint.
\end{lemma}

\begin{proof} Let $S_i$ be the path obtained by concatenating $P_i$ and $Q_i$ (leaving out the duplicate edge). We wish to show that $\mathcal{S} = \{S_i : 1 \leq i \leq k\}$ is a semi-edge-disjoint set of paths. We observe that for each $i$, every edge-adjacency appearing in $S_i$ appears in exactly one of $P_i$ or $Q_i$. Similarly, each pair $P_i, Q_j$ share at most one common edge, so no edge-adjacency appears in both $P_i$ and $Q_j$. Since $\mathcal{P}$ and $\mathcal{Q}$ are both semi-edge-disjoint, it thus follows that $S$ is a set of semi-edge-disjoint paths.
\end{proof}

We now prove the main result of this section.

\begin{proof}\emph{(Theorem \ref{mH})} Let $H$ be a graph such that $G=L(H)$ has a $K_t$-immersion. We must show that $L(mH)$ has a $K_{mt}$-immersion.

Since $L(H)$ immerses $K_t$, it follows that there \mg{exist} $t$ terminal edges in $H$ which are joined by a set $\mathcal{P}$ of semi-edge-disjoint paths joining each pair (hereafter we abbreviate semi-edge-disjoint as s.e.d.). Denote this set of edges by $S$. We note now that in $mH$, it would be sufficient to find $mt$ terminal edges and a corresponding s.e.d. set of paths (with a path between each pair). Denote by $S_m$ the set of edges in $mG$ formed by choosing all $m$ copies of $e$ for each $e \in S$. We claim that $S_m$ is a suitable set of terminal edges in $mG$.

Let $e \neq f \in S$, and let $P(e,f)\in \mathcal{P}$ be the path between them with edges $e = e_1, e_2, ..., e_{\ell} = f$. Each of $e, f$ give rise to $m$ edges in $S_m$. Denote these edges by $e^1, e^2, ..., e^m$ and $f^1, f^2, ..., f^m$ respectively. We will now define a set of s.e.d. paths joining each $e^i$ and $f^j$. If $e$ and $f$ are adjacent in $G$ (i.e. $\ell = 2$), then $e^i$ and $f^j$ are adjacent for all $1 \leq i, j \leq k$ and we are done. Thus, we may assume that $P(e,f)$ contains at least three edges, i.e. $\ell \geq 3$. For each $1 < i < \ell$, we denote the $m$ edges in $mG$ corresponding to $e_i$ by $e_{i}^1, e_{i}^2,...,e_{i}^m$. We observe that if $\ell$ is even, then for each $1 \leq i, j \leq k$, we can choose the path with edges $e^i = e_1^i,e_{2}^j,e_{3}^i,e_{4}^j,...,e_{\ell-1}^i,e_{\ell}^j = f^j$ to be the path connecting $e^i$ and $f^j$. As a set, these paths are s.e.d., since adjacencies between ``level $i$'' and ``level $j$'' of the edge-multiplicities occurs only in the path between $e^i$ and $f^j$. Hence we may assume that $\ell$ is odd.

Let $L$ be a latin square of order $m$.  For each $1 \leq i,j \leq k$, choose the path between $e^i$ and $f^j$ formed by the edges $e^i = e_1^i, e_2^j, e_3^i, e_4^j,..., e_{\ell-2}^i, e_{\ell-1}^{L(i,j)},e_{\ell}^j = f^j$; call this set of paths $\mathcal{Q}$.  If we consider the set of paths of the form $e_1^i, e_2^j, e_3^i, e_4^j,...,e_{\ell-2}^i$, this set is s.e.d., since as above, adjacencies between ``level $i$'' and ``level $j$'' of the edge-multiplicities occur only in one of the paths. We can also observe that the set of 2-edge-paths of the form $e_{\ell-2}^i, e_{\ell-1}^{L(i,j)}$ are also s.e.d., as are the set of 2-edge-paths of the form $e_{\ell-1}^{L(i,j)}, e_{\ell}^j$. The first set of 2-edge-paths are s.e.d. because of the row condition of latin squares: note that the $i, j$ 2-edge-path is the only one containing an adjacency between level $i$ of $e_{\ell-2}$ and level $L(i,j)$ of $e_{\ell-1}$, because $L(i,j)$ occurs only once in row $i$ of $L$. Similarly, the second set of 2-edge-paths are s.e.d. because of the column condition of latin squares. By applying Lemma \ref{concat} twice, we see that $\mathcal{Q}$ is a set of s.e.d. paths.

We have now found a set of s.e.d. paths between each of $\{e^1,...,e^m\}$ and $\{f^1,...,f^m\}$. We may repeat this procedure to find a set of  s.e.d. paths between any pair of terminal edges in $S_m$ that correspond to distinct terminal edges in $S$. Given a single terminal edge in $S$, there is a trivial set of s.e.d. paths joining all the corresponding terminal edges in $S_m$, with each path just consisting of the two edges in question. It is left to observe that, taken as a whole, this entire set of paths is s.e.d. Suppose instead that there exists a pair of edges in $mH$ appearing consecutively in two paths. Suppose that these paths are those connecting $a^i$ to $b^j$ and $c^k$ to $d^w$ for some $a, b, c, d\in S$, with $a, b, c, d$ not necessarily distinct but with $\{a,b\} \neq \{c,d\}$ (as we have ensured the two paths are s.e.d. otherwise). Since we obtained these paths by using the underlying paths $\mathcal{P}$ between edges in $S$, a pair of edges appearing consecutively here would correspond to a pair of edges in $H$ appearing consecutively in two different paths from $\mathcal{P}$, both the path from $a$ to $b$ and the path from $c$ to $d$. If $a=b$ or $c=d$ such paths do not even exist (we did not include paths from a terminal edge to itself in $\mathcal{P}$), and hence this is impossible. Otherwise, we get a contradiction to the fact that $\mathcal{P}$ is a set of s.e.d. paths.
\end{proof}

We now note that a minor analog of Theorem \ref{mH} is trivial. That is, if $L(H)$ contains $K_t$ as a minor, then $L(mH)$ trivially contains $K_{mt}$ as a minor. This is because the fact that $L(H)$ contains $K_t$ as a minor means, as we saw in the last section, that $H$ contains $t$ connected subgraphs $A_1, A_2, \ldots, A_t$, each with at least one edge, that are pairwise edge-disjoint but pairwise have at least one vertex in common. This extends to a selection of $mt$ connected subgraphs in $mH$ with the same properties simply by selecting the $m$ copies of $A_i$ for each $1 \leq i \leq t$.\\

\section{Clique immersions vs clique minors in line graphs}

We first prove the following helpful lemma.

\begin{lemma}\label{twocycles}
Suppose that $H$ contains two cycles, at most one of which has length two, and suppose that these two cycles share at least one common edge. Then $L(H)$ has both a $K_4$-immersion and a $K_4$-minor.
\end{lemma}

\begin{figure}
\centering
\begin{tikzpicture}

  \node[vertex] (v1) at (2,3) {};
  \node[vertex] (v2) at (0,3)  {};
  \node[vertex] (v3) at (0,0)  {};
  \node[vertex] (v4) at (2,0) {};
  \node[vertex] (v5) at (4,0)  {};
  \node[vertex] (v6) at (4,3) {};
  \node[vertex] (v7) at (2,1.5)  {};

  \node[label] (l1) at (2,0.75) {$\vdots$};
  \node[label] (l2) at (0,0.75) {$\vdots$};
  \node[label] (l3) at (0,1.5) {$\vdots$};
  \node[label] (l4) at (0,2.25) {$\vdots$};
  \node[label] (l5) at (4,0.75) {$\vdots$};
  \node[label] (l6) at (4,1.5) {$\vdots$};
  \node[label] (l7) at (4,2.25) {$\vdots$};

  \foreach \from/\to in {v1/v2,v3/v4,v4/v5,v1/v6,v1/v7}
	\path[edge] (\from) -- (\to);

  \node[label] (e) at (2,2.25) {$e$};
  \node[label] (e1) at (1,3) {$e_1$};
  \node[label] (e2) at (3,3) {$e_2$};
  \node[label] (f1) at (1,0) {$f_1$};
  \node[label] (f2) at (3,0) {$f_2$};

\end{tikzpicture}
\caption{Two cycles sharing at least one edge, as in the proof of Lemma \ref{twocycles}}
\label{twocyclespic}
\end{figure}
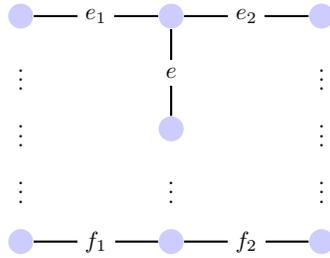

\begin{proof} We may choose two cycles $C_1$ and $C_2$ (not both of which have length two) which contain a common path $P$ of length at least one, but otherwise have no common edges. Let $e$ be the first edge in $P$. Then there exist edges $e_1 \in C_1\setminus C_2$, $e_2 \in C_2\setminus C_1$ that are adjacent to $e$. Similarly, the last edge of $P$ (which may be $e$) is adjacent to edges $f_1 \in C_1\setminus C_2$, $f_2 \in C_2\setminus C_1$. See Figure \ref{twocyclespic}. Since at most one of $C_1, C_2$ has length two, we may assume, without loss of generality, that $e_1 \neq f_1$. A suitable selection of four terminal edges for the immersion is then $e, e_1, e_2,$ and $f_1$. We note that $e, e_1$, and $e_2$ are all adjacent to each other, so we need only connect each to $f_1$ along some s.e.d. set of paths. The path $P$ and the paths corresponding to $C_1 \setminus P$, and $f_1 \cup C_2 \setminus P$ achieve this.  To see that $L(H)$ has a $K_4$-minor, we need only select as our subgraphs the three edges $e, e_1, e_2$ and then, for the fourth subgraph, the rest of the edges of $H$.
\end{proof}

We can now provide the following characterization for $K_4$-immersions and $K_4$-minors in line graphs. 

\begin{theorem}\label{K4} \mg{Let $G$ be a line graph. Then $G$ has a $K_4$-immersion iff it has a $K_4$-minor iff for all graphs $H$ such that $G = L(H)$, either $\Delta(H) \geq 4$ or $H$ contains two different cycles that share at least one edge, where at most one of these cycles has length two.}
\end{theorem}

Before we prove Theorem \ref{K4}, note that we can actually add ``iff it has a topological $K_4$-minor'' to its statement for free. This is because, for any graph $F$ with $\Delta(F)\leq 3$, a graph $G$ has an $F$-minor iff if has a topological $F$-minor, as contraction can be mimicked by vertex-suppression in this case (see, for example, page 269 of \cite{bondy}). The same is not true if we replace ``$G$ has an $F$-minor'' with ``$G$ has an $F$-immersion'', even if $F=K_4$. For example, the graph in Figure \ref{K4ex} immerses $K_4$, but does not contain $K_4$ as a minor. (To see the $K_4$-immersion, split off the pair of edges $uv$ \mg{and} $vw$, and then continue to split off along the resulting 3-path from $x$ to $z$). Hence the assumption of ``line graph'' in Theorem \ref{K4} (and Theorem \ref{minorimm}) is indeed crucial. The graph in Figure \ref{K4ex} is of course not a line graph, as it contains an induced $K_{1,3}$.

\begin{figure}
\centering
\begin{tikzpicture}

  \node[vertex] (v1) at (0,1) {x};
  \node[vertex] (v2) at (2,3)  {v};
  \node[vertex] (v3) at (4,1)  {z};
  \node[vertex] (v4) at (2,1) {y};
  \node[vertex] (v5) at (0,3)  {u};
  \node[vertex] (v6) at (4,3) {w};

  \foreach \from/\to in {v1/v2, v2/v3, v3/v4, v4/v1, v2/v4, v1/v5, v5/v2, v2/v6, v6/v3}
	\path[edge] (\from) -- (\to);

\end{tikzpicture}
\caption{A graph (which is not a line graph) that contains a $K_4$ immersion but no $K_4$ minor.}
\label{K4ex}
\end{figure}
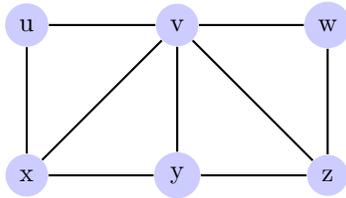

\begin{proof} We first note the truth of all the backwards implications. \mg{If $G$ contains a $K_4$-minor, then it must also contain a topological $K_4$-minor by our above comments. Hence we immediately get that $G$ has a $K_4$-immersion as well. If $\Delta(H) \geq 4$, then $G$ contains $K_4$ as a subgraph and so trivially as an immersion and minor as well. If $H$ has instead the cycle condition, then $G$ contains a $K_4$-minor and a $K_4$-immersion by Lemma \ref{twocycles}.}

Suppose now that \mg{for some $H$ such that $G = L(H)$,} $\Delta(H) \leq 3$ and that no two cycles in $H$ share an edge unless both are of length two. To complete our proof, it suffices to show that \mg{$G$} has no $K_4$-immersion. We may assume that $H$ has no pair of cycles sharing an edge, since this would correspond to a parallel edge of multiplicity three, which would have to be all of $H$ given our degree restriction (we can assume that $H$ is connected). Two cycles in $H$ cannot share a vertex either, as otherwise such a vertex would have degree at least 4. It follows that $H$ can be obtained from a tree $T$ by replacing some number of the vertices in $T$ by cycles, such that if vertex $v\in T$ is replaced by a cycle $C_v$, then each vertex $w$ adjacent to $v$ in $T$ is adjacent to exactly one vertex in $C_v$  (and if $w$ is also replaced by $C_w$, then there is exactly one edge joining $C_v$ and $C_w$).

Suppose, for a contradiction, that \mg{$G$} has a $K_4$-immersion. Then there exists a set $S$ of four terminal edges in $H$ and a  s.e.d. set of paths between them. We first observe that all four edges in $S$ cannot come from a single cycle in $H$. The only paths between edges of any cycle in $H$ lie along that cycle, and a cycle (which is its own line graph) does not immerse $K_4$. So, there are edges $e_1, e_2, e_3\in S$ such that \mg{neither $e_2$ nor $e_3$ lies on a common cycle with $e_1$. Hence, subject to renaming $e_1,e_2,e_3$, there is a cut edge $e$ such that both of the s.e.d. paths from $e_1$ to $\{e_2, e_3\}$ must contain $e$ (we ``swap" $e_1$ with either $e_2$ or $e_3$ if $e_1$ lies on a cycle that intersects a path between $e_2$ and $e_3$). Note that it is possible that $e_1=e$.}  Since $\Delta(H) \leq 3$, each end of $e$ is incident with at most two other edges in $H$. Hence while there can be two s.e.d. paths from $e_1$ to $\{e_2, e_3\}$, the fourth terminal edge $e_4$ cannot be on the same side of the edge-cut as $\{e_2, e_3\}$. Hence we may assume that the fourth terminal edge is on the opposite side of the edge-cut (note that we may avoid the case where the fourth edge is $e$ by simply choosing $e_1=e$ in this case). However now both s.e.d. paths from $e_4$ to $\{e_2, e_3\}$ must contain $e$, and these cannot be s.e.d. from the two paths from $e_1$ to $\{e_2, e_3\}$. Hence \mg{$G$} does not immerse $K_4$.
\end{proof}

We are now able to prove our final result of this paper, which we restate now for convenience.

\setcounter{section}{1}
\setcounter{theorem}{3}
\begin{theorem} Let $G$ be a line graph. If $t\leq 4$, then $G$ has a $K_t$-immersion iff $G$ has a $K_t$-minor. When \mg{$t\geq5$} this equivalence fails in both directions: there exists a line graph with a \mg{$K_t$}-immersion but no \mg{$K_t$}-minor, and there exists a line graph with a \mg{$K_t$}-minor but no \mg{$K_t$}-immersion.
\end{theorem}

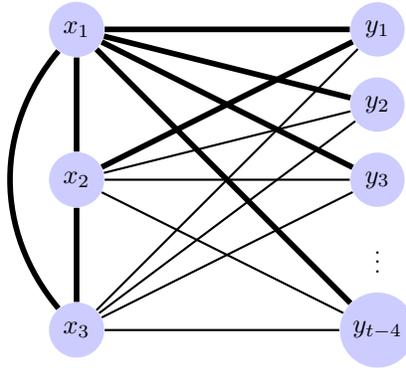
\begin{figure}
\centering
\begin{tikzpicture}

  \node[vertex] (x1) at (0,4) {$x_1$};
  \node[vertex] (x2) at (0,2)  {$x_2$};
  \node[vertex] (x3) at (0,0)  {$x_3$};

  \path[thickedge] (x1) to [bend right = 40] (x3);
  \path[thickedge] (x1) to (x2);
  \path[thickedge] (x2) to (x3);

  \node[vertex] (y1) at (4,4) {$y_1$};
  \node[vertex] (y2) at (4,3) {$y_2$};
  \node[vertex] (y3) at (4,2) {$y_3$};
  \node[vertex] (y) at (4,0) {$y_{t-4}$};

  \node[label] at (4,1) {$\vdots$};

  \foreach \from/\to in {x2/y2,x2/y3,x2/y,x3/y1,x3/y2,x3/y3,x3/y}
	\path[edge] (\from) -- (\to);

  \foreach \from/\to in {x1/y1,x1/y2,x1/y3,x1/y,x2/y1}
	\path[thickedge] (\from) -- (\to);

\end{tikzpicture}
\caption{A graph $H$ such that $L(H)$ has a $K_t$-immersion but not a $K_t$-minor for $t \geq 5$.}
\label{K5ex}
\end{figure}

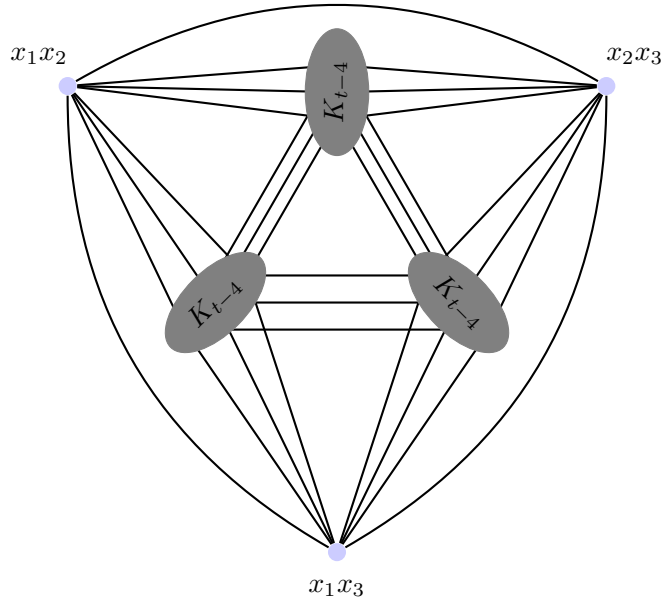
\begin{figure}
\centering
\begin{tikzpicture}[scale = 1.2, every node/.style={scale=1.2}]

	\node[Svertex] (a1) at (0,0.6) {};
	\node[Svertex] (a2) at (0,0.3) {};
	\node[Svertex] (a3) at (0,0) {};

	\node[Svertex] (b1) at (-1.69279,-2.332) {};
	\node[Svertex] (b2) at (-1.34638,-2.032) {};
	\node[Svertex] (b3) at (-1,-1.732) {};

	\node[Svertex] (c1) at (1.69279,-2.332) {};
	\node[Svertex] (c2) at (1.34638,-2.032) {};
	\node[Svertex] (c3) at (1,-1.732) {};

	\node[Svertex] (a) at (0,-4.8) {};
	\node[Svertex] (b) at (2.983746,0.368) {};
	\node[Svertex] (c) at (-2.983746,0.368) {};

	\foreach \from/\to in {a1/b1,b1/c1,c1/a1,a2/b2,b2/c2,c2/a2,a3/b3,b3/c3,c3/a3,a/b1,a/b2,a/b3,a/c1,a/c2,a/c3,b/a1,b/a2,b/a3,b/c1,b/c2,b/c3,c/a1,c/a2,c/a3,c/b1,c/b2,c/b3}
		\path[edge] (\from) -- (\to);

	\draw[rotate around= {45:(-1.34638,-2.032)},gray, fill = gray, fill opacity = 1] (-1.34638,-2.032) ellipse (20pt and 10pt);
	\draw[rotate around= {135:(1.34638,-2.032)},gray, fill = gray, fill opacity = 1] (1.34638,-2.032) ellipse (20pt and 10pt);
	\draw[rotate around= {90:(0,0.3)},gray, fill = gray, fill opacity = 1] (0,0.3) ellipse (20pt and 10pt);

	\node[label2,rotate = 45] at (-1.34638,-2.032) {$K_{t-4}$};
	\node[label2,rotate = 315] at (1.34638,-2.032) {$K_{t-4}$};
	\node[label2,rotate = 90] at (0,0.3) {$K_{t-4}$};

	\path[edge] (a) to [bend right = 30] (b);
	\path[edge] (b) to [bend right = 30] (c);
	\path[edge] (c) to [bend right = 30] (a);

	\node[label] at (-3.3,0.7) {$x_1x_2$};
	\node[label] at (3.3,0.7) {$x_2x_3$};
	\node[label] at (0, -5.2) {$x_1x_3$};

\end{tikzpicture}
\caption{A line graph containing a $K_t$-immersion but not a $K_t$-minor for $t \geq 5$. Each of $x_1x_2, x_2x_3, x_1x_3$ is joined completely to two of the $K_{t-4}$ cliques. Furthermore, there are $t-4$ disjoint triangles between the $K_{t-4}$ cliques.}
\label{K5exLine}
\end{figure}

\begin{proof}
The characterization in Theorem \ref{K4} says that $G$ has a $K_4$-immersion if and only if it has a $K_4$-minor. Since $G$ is necessarily simple, we note that it contains a $K_3$-minor iff it contains a $K_3$-immersion iff it contains a cycle (of any length). For $t =1, 2$, the required equivalence is trivial.

Consider the graph $H$ pictured in Figure \ref{RScounter}. As discussed in Section 2, $L(H)$ contains $K_5$ as a minor but not as an immersion. (Moreover, as we noted in Section 2, this $H$ can be easily modified to yield a line graph containing a $K_t$-minor but no $K_t$-immersion for any $t\geq 5$).


\mg{Consider now the graphs pictured in Figures \ref{K5ex} and \ref{K5exLine}; note that the latter is $L(H)$ where $H$ is the graph in Figure \ref{K5ex}. Moreover, $L(H)$ is a subgraph of a graph described and shown to be $K_t$-minor free by Cames van Batenburg et al. in \cite{counterex}.  In particular, in Proposition 13 of \cite{counterex} choose $k=t$ and $\Delta = k+1$ and note that $L(H)$ is clearly a subgraph of $G_{k,\Delta}$ as pictured in Figure 1 of \cite{counterex}. }

\mg{It remains now only to show that $L(H)$ has a $K_t$-immersion. To this end, let $X = \{x_1,x_2,x_3\}$ and $Y = \{y_1,y_2,...,y_{t-4}\}$ as in Figure \ref{K5ex}. Select as terminal edges the three edges induced by $X$, the $t-4$ edges joining $x_1$ to $Y$, and the edge $x_2y_1$ (any edge joining $x_2$ to a vertex of $Y$ would suffice). These edges have been darkened in Figure \ref{K5ex}. We claim that the necessary s.e.d. set of paths between these edges exists. Many of these paths are taken care of by immediate edge adjacencies. The remaining paths needed are the following: paths between $x_1y_i$ and $x_2x_3$ for each $1 \leq i \leq t-4$, paths between $x_1y_j$ and $x_2y_1$ for each $2 \leq j \leq t-4$, and a path between $x_1x_3$ and $x_2y_1$. For each $1 \leq i \leq t-4$, choose the path with edges $x_1y_i,x_3y_i,x_2x_3$ and denote the collection of such paths by $\mathcal{P}_I$. For each $2 \leq j \leq t-4$, choose the path with edges $x_1y_j,x_2y_j,x_2y_1$ and denote the collection of such paths by $\mathcal{P}_J$. Finally, choose the path $P^*$ with edges $x_1x_3,x_3y_1,x_2y_1$. We first note that no two terminal edges appear consecutively in any of these new paths, so they are s.e.d. with the edge adjacencies already being used. Furthermore, if $P_1$ and $P_2$ are distinct paths in $\mathcal{P}_I \cup \mathcal{P}_J \cup \{P^*\}$, then we observe that $P_1$ and $P_2$ have at most one edge in common. It follows that $\mathcal{P}_I \cup \mathcal{P}_J \cup \{P^*\}$ is s.e.d., and coupled with the necessary edge adjacencies this yields a complete s.e.d. set as needed.  }

\end{proof}

\bibliographystyle{amsplain}

\end{document}